\documentclass[12pt]{article}
\usepackage{amsmath}%
\usepackage{mathrsfs}
\usepackage{amsfonts}%
\usepackage{amssymb}%
\usepackage{graphicx}
\newtheorem{theorem}{Theorem}

\newtheorem{lemma}[theorem]{Lemma}

\newenvironment{proof}[1][Proof]{\textbf{#1.} }{\ \rule{0.5em}{0.5em}}
\numberwithin{equation}{section}

\title { A note on the Ramanujan's
master theorem %and Cahen-Mellin integral
}
\author{ Lazhar Bougoffa \thanks{IIUM University, Faculty of  Science, Department of Mathematics,
Riyadh, Saudi Arabia.
E-mail address: lbbougoffa@imamu.edu.sa}}
\begin{document}
\maketitle
\begin{abstract}
In this note, it is shown  that the \emph{Ramanujan's
Master Theorem (RMT)} when $n$  is a positive integer  %and  \emph{Euler integral  representation of the gamma function} %and \emph{ Cahen-Mellin integral}
 can be obtained, as  a special case, from a new integral  formula. Furthermore, we give a simple proof of the RMT when $n$ is not an  integer.
 %this formula can be used to quickly
%evaluate certain integrals not expressible in terms of elementary
%functions.
\end{abstract}
\bigskip
%\noindent{\bf AMS Classification:}
\maketitle
\section{Introduction}
In this note, we prove a new integral formula for the evaluation of definite
integrals and
show that   the Ramanujan's Master Theorem
(RMT) \cite{1, 2} when $n$  is a positive integer %and Euler integral  representation of the gamma function
can be easily derived, as a special case, from this  integral formula. This formula can be used to quickly
evaluate certain integrals not expressible in terms of elementary
functions. For $n$ is not an integer, we shall also give a simple proof of the RMT.
\section{Main result}
To clarify the procedure, we begin by considering the following
Cauchy-Frullani integral \cite{3}:
\begin{lemma}
Let $f$ be a continuous function and assume that both $f(\infty)$
and $f(0)$  exist. Then
\begin{equation}
\label{2}\int_{0}^{\infty}\frac{f(\alpha x)-f(\beta
x)}{x}dx=\left(f(\infty)-f(0)\right)\ln \frac{\alpha}{\beta}, \
\alpha, \ \beta>0.
\end{equation}
\end{lemma}
This formula was first published by Cauchy in 1823, and more
completely in 1827 with a beautiful proof.\\
The following lemma is a new helpful tool in the proof of the Ramanujan's Master Theorem \cite{1,2} and other integrals.
\begin{lemma}
Let $f\in \mathbb{C}^{n}([0, \infty))$ such that both $f(\infty)$ and
$f(0)$  exist. Then
\begin{equation}
\label{5}\int_{0}^{\infty}x^{n-1}f^{(n)}(x)
dx=(-1)^{n-1}\left[f(\infty)-f(0)\right]\Gamma(n), \ \Gamma(n)=(n-1)!.
\end{equation}
\end{lemma}
\begin{proof}
Differentiating both sides of  Eq.(\ref{2}) in Lemma 1 $n-$times with respect
to $\alpha,$ and using the chain rule $\dfrac{d}{d\alpha} f(\alpha
x)=\dfrac{d}{d(\alpha x)} [f(\alpha x)]\times\dfrac{d(\alpha
x)}{d\alpha },$ we obtain
\begin{equation}
\label{6}\int_{0}^{\infty}x^{n-1}\frac{d^{n}}{d(\alpha x)^{n}}
\left[f(\alpha
x)\right]dx=(-1)^{n-1}\left[f(\infty)-f(0)\right]\frac{(n-1)!}{\alpha^{n}},
\ \alpha>0.
\end{equation}
The change of variable $t=\alpha x$ in the LHS of (\ref{6}) yields
\begin{equation}
\label{7}\frac{1}{\alpha^{n}}\int_{0}^{\infty}t^{n-1}\frac{d^{n}f(t)}{dt^{n}}
dt=(-1)^{n-1}\left[f(\infty)-f(0)\right]\frac{(n-1)!}{\alpha^{n}}, \
\alpha>0.
\end{equation}
The proof is complete.
\end{proof}
\section{Applications}
\subsection{The Ramanujan's
Master Theorem} The Ramanujan's Master Theorem \cite{1, 2} states
that
\begin{theorem}
If $F(x)$ is defined through the series expansion $F(x)=
\sum_{k=0}^{\infty}\phi(k)\frac{(-x)^{k}}{k!},$ with $\phi(0)\neq
0.$ Then
\begin{equation}
\label{13}\int_{0}^{\infty}x^{n-1}\sum_{k=0}^{\infty}\phi(k)\frac{(-x)^{k}}{k!}
dx=\Gamma(n)\phi(-n),
\end{equation}
where $n$ is a positive integer.
\end{theorem}
It was widely used by the indian mathematician Srinivasa Ramanujan
(1887-1920)  to calculate definite integrals and infinite series.\\
Ramanujan asserts that his proof is legitimate with just simple assumptions \cite{1, 2}:
$(1)$ $F(x)$ can be expanded in a Maclaurin series;
$(2)$ $F(x)$ is continuous on $(0,\infty);$
$(3)$ $n>0;$ and $(4)$ $x^{n}F(x)$ tends to $0$ as $x$ tends to $\infty.$
\\
We note  below that the Ramanujan's Master Theorem  can be derived as a special case from (\ref{5}) when $n$ is a positive integer.\\
\begin{proof}(\textbf{Using (\ref{5})})
Assume that $f(x)$ is expanded in a Maclaurin series $f(x)=\sum_{k=0}^{\infty}\psi(k)\frac{(-x)^{k}}{k!},$ where  $f(0)= \psi(0)\neq 0$
and $f(x)$ tends to $0$ as $x$ tends to $\infty.$\\ A simple computation leads to  $f^{(n)}(x)=(-1)^{n}
\sum_{k=0}^{\infty}\psi(n+k)\frac{(-x)^{k}}{k!}.$\\ Substituting into  (\ref{5}), we obtain
\begin{equation}
\label{14}\int_{0}^{\infty}x^{n-1}\sum_{k=0}^{\infty}\psi(n+k)\frac{(-x)^{k}}{k!}
dx=f(0)\Gamma(n)=\psi(0)\Gamma(n).
\end{equation}
We see that, in the notation of the  Ramanujan's Master Theorem,
$\phi(k)=\psi(n+k), \ k=0,1,...$ and hence  $\phi(-n)=\psi(0), \ n \in \mathbb{N}.$ \\ This is precisely formula (\ref{13}), and the proof is complete.
\end{proof}
\subsection{Other integrals involving special functions}
\subsubsection{The Euler integral}
An immediate consequence of  (\ref{5}) is the evaluation of the following integral.
\begin{equation}
\label{15}\int_{0}^{\infty} x^{n-1}e^{-ax}dx=a^{-n}\Gamma(n), \ a>0.
\end{equation}
This integral is known as the  Euler integral representation of the gamma function. It was considered by Euler in 1729 and 1730  \cite{3}.
\\ This follows simply by letting     $f(x)=e^{-ax},\ f(0)=1, \ f(\infty)=0$ and $f^{(n)}(x)=(-a)^{n}e^{-ax}$ in (\ref{5}).
\subsection{Integral representation of the beta function}
The beta function $B(n,m)$  is defined by \cite{3}
\begin{equation}
B(n,m)=\int_{0}^{\infty}x^{n-1}\frac{1}{(1+x)^{n+m}}
dx=\frac{\Gamma(n)\Gamma(m)}{\Gamma(n+m)}, \ m,n=1,2,...,.
\end{equation}
This follows simply by letting    $f(x)=\frac{1}{(1+x)^{m}}, \
f(\infty)=0, \ f(0)=1$ and
$f^{(n)}(x)=(-1)^{n}m(m+1)...(m+n-1)\frac{1}{(1+x)^{n+m}}, \
n=1,2,...$ in (\ref{5}), and using the above property of the gamma
function.
\subsubsection{Gaussian integral }
\begin{equation}
\int_{0}^{\infty} e^{-x^{2}}dx=\frac{\sqrt{\pi}}{2}.
\end{equation}
This follows simply by letting   $f(x)=erf(x),
f'(x)=\frac{2}{\sqrt{\pi}}e^{-x^{2}}, \ erf(\infty)=1, \ n=1$  and
$erf(0)=0$ in (\ref{5}).
\subsubsection{Integral involving Hermite polynomials $H_{n}(x)$}
\begin{eqnarray}
\int_{0}^{\infty}x^{n-1}H_{n-1}(x)e^{-x^{2}}dx=\frac{\sqrt{\pi}}{2}\Gamma(n).
\end{eqnarray}
This follows simply by letting    $f(x)=erf(x)$ in (\ref{5}) and
using the Rodrigues formula for the Hermite polynomials:
\begin{eqnarray}
\frac{d^{n}f(x)}{dx^{n}}\left[erf(x)\right]=(-1)^{n-1}\frac{2}{\sqrt{\pi}}H_{n-1}(x)e^{-x^{2}}.
\end{eqnarray}
\subsubsection{Integral involving Laguerre polynomials $L_{n}(x)$}
\begin{eqnarray}
\int_{0}^{\infty}x^{n-1}L_{n}(x)e^{-x}dx=0.
\end{eqnarray}
where $L_{n-1}(x)$ are Laguerre polynomials.
This follows simply by
letting    $f(x)=x^{n}e^{-x}, \ f(\infty)=0= f(0)$  in (\ref{5})
and using the Rodrigues formula for the Laguerre polynomials:
\begin{eqnarray}
\frac{d^{n}f(x)}{dx^{n}}\left[x^{n}e^{-x}\right]=n!L_{n}(x)e^{-x}.
\end{eqnarray}
\section{A simple proof of the RMT when $n$ is not an  integer}
We now give a simple proof of the RMT when $n$ is not an  integer.\\
We recall that the Mellin transform is the integral transform defined by %\cite{5,6,7}
\begin{equation}
\mathscr{M}\{f(t), \ s \}=\int_{0}^{\infty}t^{s-1}f(t)dt,
\end{equation}
where $s$ is a complex number.\\
Also, the change of variable $t=e^{-x}$ transforms $\mathscr{M}\{f(t), \ s \}$ into  the two-sided Laplace transform of $f(e^{-x}).$ This can be written  as
\begin{equation}
\mathscr{M}\{f(t), \ s \}=
\mathscr{L}\{f(e^{-x}),\ s \}=\int_{-\infty}^{\infty}e^{-sx}f(e^{-x})dx.
\end{equation}
\begin{enumerate}
  \item Let
\begin{eqnarray}
f(x)=\left\{
\begin{array}{rcl}
\sum_{k=0}^{\infty}\frac{\phi(k)}{k!}(-x)^{k}, \  x\geq 0,\\
0, \ x<0.
\end{array}
\right.
\end{eqnarray}
Thus
\begin{equation}
\int_{0}^{\infty}x^{s-1}\sum_{k=0}^{\infty}\phi(k)\frac{(-x)^{k}}{k!}
dx=\int_{0}^{\infty}e^{-sx}\sum_{k=0}^{\infty}\phi(k)\frac{(-1)^{k}}{k!}e^{-kx}dx.
\end{equation}
Since $\mathscr{L}\{e^{-kx},\ s \}= \frac{1}{s+k}, \ \Re(s)>-k.$ Therefore,
\begin{equation}
\int_{0}^{\infty}x^{s-1}\sum_{k=0}^{\infty}\phi(k)\frac{(-x)^{k}}{k!}
dx=\sum_{k=0}^{\infty}\phi(k)\frac{(-1)^{k}}{k!}\frac{1}{s+k}.
\end{equation}
We recall that from the well-known functional equation $\Gamma(s+1)=s\Gamma(s),$ we have
\begin{equation}
\Gamma(s)=\frac{\Gamma(s+m+1)}{s(s+1)...(s+m)}.
\end{equation}
Thus $\Gamma(s)$ has poles at $s=-m, \ m=0,1,2,...$\\ Thus $\lim_{s\rightarrow -m}(s+m)\Gamma(s)=\frac{(-1)^{m}}{m!}$ as $s\rightarrow -m.$  Hence
$\Gamma(s)\sim \frac{(-1)^{m}}{m!}\frac{1}{s+m}.$
Consequently,
\begin{equation}
\phi(-s)\Gamma(s)\sim \phi(m)\frac{(-1)^{m}}{m!}\frac{1}{s+m} \ \mbox{as} \ s\rightarrow -m.
\end{equation}
This means that $\phi(m)\frac{(-1)^{m}}{m!}\frac{1}{s+m}$ is a singular element of the function $\phi(-s)\Gamma(s)$  at $s=-m.$ From the definition of the singular expansion
of $\phi(-s)\Gamma(s),$ we obtain
\begin{equation}
\phi(-s)\Gamma(s)\asymp \sum_{k=0}^{\infty}\phi(k)\frac{(-1)^{k}}{k!}\frac{1}{s+k}
\end{equation}
and the proof of  Ramanujan's Master Theorem is complete.
  \item Let
\begin{eqnarray}
f(x)=\left\{
\begin{array}{rcl}
\sum_{k=0}^{\infty}\phi(k)(-x)^{k}, \  x> 0,\\
0, \ x\leq 0.
\end{array}
\right.
\end{eqnarray}
Thus
\begin{equation}
\int_{0}^{\infty}x^{s-1}\sum_{k=0}^{\infty}\phi(k)(-x)^{k}
dx=\sum_{k=0}^{\infty}(-1)^{k}\phi(k)\frac{1}{s+k}.
\end{equation}
Proceeding as before,  we have
\begin{equation}
\phi(-s)(-s)!\Gamma(s)\sim \phi(m)(-1)^{m}\frac{1}{s+m} \ \mbox{as} \ s\rightarrow -m.
\end{equation}
This means that $\phi(m)(-1)^{m}\frac{1}{s+m}$ is a singular element of the function $\phi(-s)(-s)!\Gamma(s).$ From the definition of the singular expansion
of $\phi(-s)(-s)!\Gamma(s),$ we obtain
\begin{equation}
\phi(-s)(-s)!\Gamma(s)\asymp \sum_{k=0}^{\infty}(-1)^{k}\phi(k)\frac{1}{s+k}.
\end{equation}
Using the well-known property $(-z)!\Gamma(z)=\frac{\pi}{\sin \pi z}, \ z\neq 0, \pm1, \pm2,....,$ we get
\begin{equation}
\int_{0}^{\infty}x^{s-1}\sum_{k=0}^{\infty}\phi(k)(-x)^{k}
dx=\frac{\pi}{\sin \pi s}\phi(-s),
\end{equation}
which is the Hardy's version of the the RMT (Theorem (Hardy))\cite{2}.
\end{enumerate}

\end{document}